\documentclass{article}[12pt]

\usepackage[english]{babel}

\usepackage{a4wide}
\usepackage{amsmath}
\usepackage{graphicx}
\usepackage{amsthm}
\usepackage{xspace}
\usepackage{complexity}
\usepackage{amsthm,amssymb}
\usepackage{doi}
\usepackage[capitalise]{cleveref}
\usepackage{comment}

\newtheorem{theorem}{Theorem}
\newtheorem{lemma}[theorem]{Lemma}

\usepackage{authblk}

\author[1]{Carl Feghali}
\author[2]{Felicia Lucke}
\newcommand{\email}[1]{
  \texttt{#1}
}

\affil[1]{Univ Lyon, EnsL, CNRS, LIP, F-69342, Lyon Cedex 07, France\protect\\\email{carl.feghali@ens-lyon.fr}}
\affil[2]{Department of Informatics, University of Fribourg, Switzerland\protect\\\email{felicia.lucke@unifr.ch}}

\title{A (simple) proof of the rna conjecture on powers of cycles}

\makeatletter
\def\@maketitle{%
\newpage%
\null%
\begin{center}%
    \let\footnote\thanks %
    {\huge \@title %
      \par
    }
  \vskip 1.5em
    {\lineskip .5em
     \begin{tabular}[t]{c}
        \baselineskip=12pt
        \@author
     \end{tabular}
     \par
    }
\end{center}
\par
\vskip 1.5em}
\makeatother

\begin{document}

\maketitle

\begin{abstract}
We verify a recent conjecture of Sehrawat, Kumar and Ahlawat on the minimum bisection width (a notion that was rediscovered in 2020 by Acharya and
Kureethara under the name of rna number) of powers of cycles. 
\end{abstract}

\section{Introduction}

Let $G$ be a graph, and let $\sigma: E(G) \rightarrow \{-1, 1\}$. The pair $(G, \sigma)$ is called a \emph{parity signed graph with respect to some labeling of the vertices of $G$}, say as $v_1, \dots, v_n$, if for every distinct $i, j \in [n]$, $\sigma((v_i, v_j)) = 1$ if $i$ and $j$ have the same parity and $\sigma((v_i, v_j)) = -1$ otherwise. The \emph{rna number} of a graph $G$, introduced in 2020 by  Acharya and
Kureethara \cite{acharya} and denoted $\sigma^{-1}(G)$, is the minimum number of edges $e$ of $G$ such that $\sigma(e) = -1$ over all parity signed graphs $(G, \sigma)$. 

Perhaps most strikingly to readers familiar with graph optimization problems, the rna number corresponds to the \emph{minimum bisection width}, a fundamental notion dating back to the 1970s and having a wide range of applications; see \cite{cygan, DPS02, DDSS24} and references within. Of relevance to this note, the rna number is known for special graph families such as, for example, stars, wheels, paths, cycles and complete graphs \cite{acharya2}. 
For arbitrary graphs $G$, Kang et al.~\cite{kang}  showed $\sigma^{-1}(G) \leq \lfloor \frac{2|E(G)|+|V(G)|}{4} \rfloor$ and characterized when equality is achieved.

Let $d \geq 2$. The $d$th power of the cycle $C_n$ on $n$ vertices, denoted $C_n^d$, is obtained from $C_n$ by adding an edge $(u, v)$ in $C_n$ whenever the distance between $u$ and $v$ in $C_n$ is at most $d$. In a recent manuscript \cite{SKA23}, it was conjectured that $\sigma^{-1}(C_n^d) = d(d+1)$ for every $d \geq 2$ and $n \geq 2d+1$. In that same paper, the inequalities $2d \leq \sigma^{-1}(C_n^d) \leq d(d+1)$ were established and the conjecture was verified for $d \in \{2, 3\}$. In this brief note, we prove the conjecture. 

\begin{theorem}\label{thm}
Let $d \geq 2$ and $n \geq 2d+1$. Then $\sigma^{-1}(C_n^d) = d(d+1)$.
\end{theorem}

\section{Proof}

Let $G$ be a graph. A \emph{balanced coloring} of $G$ is a function $f: V(G) \rightarrow \{1,2\}$ such that $|f^{-1}(2)| - 1 \leq |f^{-1}(1)| \leq |f^{-1}(2)| + 1$ (and hence $|f^{-1}(1)| - 1 \leq |f^{-1}(2)| \leq |f^{-1}(1)| + 1$). Given a balanced coloring $f$ of $G$, let $g(f) = \{(x,y) \in E(G): f(x) \not= f(y)\}$. Notice then that 
$$
\sigma^{-1}(G):=\min_f |g(f)|
$$
over all balanced colorings $f$ of $G$. 

The following two lemmas combined imply Theorem \ref{thm}. 

\begin{lemma}
Let $d \geq 2$ and $n \geq 2d+1$. Then $\sigma^{-1}(C_n^d) \leq d(d+1)$.
\end{lemma}
\begin{proof}
Let $G$ be the graph with vertex set \[V(G) = \left\{v_0, \dots, v_{n-1}\right\}\] and edge set
\[E(G) = \left\{ (v_i,v_{i+j}): i \in \left\{0 , \dots,  n-1\right\}, j \in \left\{1 , \dots,  d \right\} \right\},\]
where addition is modulo $n$. Thus, $G:= C_n^d$.

Now, let $A = \{v_i: 0 \leq i < \lfloor \frac{n-1}{2} \rfloor\}$, $B = V(G) \setminus A$ and $f: V(G) \rightarrow \{1,2\}$ be the coloring such that $f(v) = 1$ if $v \in A$ and $f(v) = 2$ if $v \in B$. Then  $|f^{-1}(2)| - 1 \leq |f^{-1}(1)| \leq |f^{-1}(2)|$ so that $f$ is balanced and $|g(f)| = 2\binom{d}{2} = d(d+1)$; thus  $\sigma^{-1}(C_n^d) \leq d(d+1)$ as needed. 
\end{proof}

\begin{lemma}
Let $d \geq 2$ and $n \geq 2d+1$. Then $\sigma^{-1}(C_n^d) \geq d(d+1)$.
\end{lemma}
\begin{proof}
Let $G$ be the graph with vertex set \[V(G) = \left\{v_0, \dots, v_{n-1}\right\}\] and edge set
\[E(G) = \left\{ (v_i,v_{i+j}): i \in \left\{0 , \dots,  n-1\right\}, j \in \left\{1 , \dots,  d \right\} \right\},\]
where here and forth all addition is modulo $n$. Thus, $G:= C_n^d$. For a contradiction,
assume that $G$ is a counterexample to the lemma, with $n$ as small as possible.
Thus, $G$ has a balanced coloring~$f$ such that \begin{equation}\label{eq:1}
 |g(f)| < d(d+1).   
\end{equation} 

Firstly, $n > 2d+1$, since otherwise $G$ is complete and so $|g(h)| = d(d+1)$ for every balanced coloring $h$ of $G$, which is a contradiction. Assume WLOG that $|f^{-1}(1)| \geq |f^{-1}(2)|$ and $f(v_0) = 1$. Let $S = \{(v_i, v_{i+d+1}): i \in \left\{n-d , \dots,  n-1\right\}\}$, and let $H = G - v_0 + S$. Then $H := C_{n-1}^d$. Let $f'$ be the coloring of $H$ given by $f'(v) = f(v)$ for each $v \in V(H)$. 

Now, by the choice of $v_0$, $f'$ is a balanced coloring of $H$; moreover, for every $(u, v) \in S$ such that $f'(u) \not= f'(v)$, vacuously either $f(v_0) \not= f(u)$ or $f(v_0) \not= f(v)$. Let $T$ denote the set of edges incident with $v_0$ in $G$. Therefore, as $v_0$ is adjacent to both $u$ and $v$, we have 
\begin{equation} \label{eq:2}
\begin{split}
|g(f')| & = |\{(u, v) \in S: f'(u) \not= f'(v)\}| + |\{(u, v) \in E(H) \setminus S: f'(u) \not= f'(v)\}|  \\
 & =  |\{(u, v) \in S: f'(u) \not= f'(v)\}| + |\{(u, v) \in E(G) \setminus T: f(u) \not= f(v)\}| \\
 & \leq |\{(v_0, u) \in T: f(v_0) \not= f(u)\}| + |\{(u, v) \in E(G) \setminus T: f(u) \not= f(v)\}| \\ & = |g(f)|,
\end{split}
\end{equation}
and so, by (\ref{eq:1}) and (\ref{eq:2}), $\sigma^{-1}(H) < d(d+1)$, and this contradiction to the minimality of $G$ completes the proof.   
\end{proof}

\bibliographystyle{alpha}
\bibliography{ref}

\end{document}